\documentclass[reqno,11pt]{amsart}
\usepackage{epsfig,amscd,amssymb,amsmath,amsfonts}
\usepackage{bigdelim}
\usepackage{amsmath}
\usepackage{graphicx}
\usepackage{amsthm,color}
\usepackage{verbatim}
\usepackage{tikz}
\usepackage{lscape}
\usetikzlibrary{graphs}
\usetikzlibrary{graphs,quotes}
\usetikzlibrary{decorations.pathmorphing}

\tikzset{snake it/.style={decorate, decoration=snake}}
\tikzset{snake it/.style={decorate, decoration=snake}}

\usetikzlibrary{decorations.pathreplacing,decorations.markings,snakes}

\newtheorem{theorem}{Theorem}[section]
\newtheorem{lemma}[theorem]{Lemma}
\newtheorem{proposition}{Proposition}[section]
\theoremstyle{definition}
\newtheorem{definition}[theorem]{Definition}
\newtheorem{corollary}[theorem]{Corollary}
\newtheorem{conjecture}[theorem]{Conjecture}

\newtheorem{problem}[theorem]{Open Problem}

\theoremstyle{remark}
\newtheorem{remark}[theorem]{Remark}

\numberwithin{equation}{section}

\usepackage[margin=1.05in]{geometry}
\usepackage[colorlinks]{hyperref}
\usepackage{siunitx}

\usepackage{blkarray, bigstrut}
\usepackage{gauss}

\makeatletter
\let\@wraptoccontribs\wraptoccontribs
\makeatother

\begin{document}

\title{Permutation-Like Matrices}

\author{Steven R. Lippold}
\address{Department of Mathematics, Taylor University, Upland, IN 46989}
\email{steve\_lippold@taylor.edu}



\begin{abstract} Permutation Matrices are a well known class of matrices which encode the elements of the symmetric group on $d$ elements as a square $d\times d$ matrix. Motivated by \cite{S2Up}, we define a similar class of matrices which are a generalization of Permutation Matrices. We give explicit formulas for the multiplication of these matrices. Lastly, we discuss the spectral radius, eigenvalues, and periodicity before giving a form of Birkhoff-Von Neumann's Theorem for Left Stochastic Matrices.
\end{abstract}

\maketitle


%

\section{Introduction}

One well known class of matrices is that of the Permutation Matrices (see for example \cite{Linear}). The Permutation Matrices are matrices with one 1 in each row and column and zeroes elsewhere. These matrices are useful for a variety of reasons. For one, they encode row and column permutation as left and right multiplication, respectively. In addition, there exists a group isomorphism between $d\times d$ Permutation Matrices under matrix multiplication and $S_d$, the symmetric group on $d$ letters. Lastly, there are a variety of algorithms that utilize Permutation Matrices and every Doubly Stochastic Matrix can be written as a convex combination of Permutation Matrices.

\cite{S2Up} discussed an identification of edge $d$-partitions of the complete graph on $2d$ verices $K_{2d}$ with $d\times d(2d-1)$ matrices. In particular, there exists a matrix multiplication on $d\times d(2d-1)$ matrices that involves splitting up each of the $d\times d(2d-1)$ matrices as an ordered collection of $2d-1$ square $d\times d$ matrices, which can then be multiplied and reconstructed as a new $d\times d(2d-1)$ matrix. In the case of edge $d$-partitions of $K_{2d}$, these smaller $d\times d$ matrices were similar to Permutation Matrices in that there was precisely one 1 in each column, with the rest of the entries zero. However, unlike Permutation Matrices, there could be more than one 1 in a given row. As such, in this paper we look to examine such matrices closer, starting with characterizing the matrix multiplication.

We would like to briefly explain the structure of the paper. In Section 2, we introduce Permutation-Like Matrices, which form a monoid. Then, we recall the symmetric group actions of row and column permutations which is useful for this paper.

In Section 3, we classify multiplication for Permutation-Like Matrices. First, we consider the case of $2\times 2$ Permutation-Like Matrices, of which there are 3 non-identity matrices. Then, we consider the case of $3\times 3$ Permutation-Like Matrices. In doing this, we give definitions for certain classes of these matrices, simplifying the possible cases. Lastly, we consider the case of $d\times d$ Permutation-Like Matrices, with our approach modelled on the case of $3\times 3$ Permutation-Like Matrices.

In Section 4, we consider several properties of Permutation-Like Matrices. First, we show that in the case $d=2$ and $d=3$, we have that all Permutation-Like Matrices are periodic or satisfy $A^2=R_m^{(d)}$, where $R_m^{(d)}$ is a special class of Permutation-Like Matrices. Then, we give an explicit description of the eigenvalues in the case $d=2$ and $d=3$ before we compute the spectral radius of all $d\times d$ Permutation-Like Matrices for $d\geq2$. Lastly, we show that every Left Stochastic Matrix can be written as a convex combination of Permutation-Like Matrices, which gives a form of the classical Birkhoff-Von Neumann's Theorem for Left Stochastic Matrices. In Section 5, we have some remarks and open questions regarding the material covered in this paper.

\section{Permutation-Like Matrices}
Throughout this paper, let $d\geq1$ and $k$ to be a field. We assume that all matrices have coefficients in $k$. For a given $d\times d$ matrix $A$, let $a_{i,j}$ denote the entry in the $i^{th}$ row and $j^{th}$ column of $A$. 

\subsection{Permutation-Like Matrices}
We will start by define a type of matrix, which previously appeared in \cite{S2Up}.
\begin{definition}
    Let $A$ be a $d\times d$ matrix such that $a_{i,j}=0$ or $1$ for all $1\leq i,j\leq d$ and, for each $1\leq j\leq d$, there exists a unique $1\leq i\leq d$ such that $a_{i,j}=1$.
    
    Then, we call $A$ a Permutation-Like Matrix (PLM). We take $PL_d$ to be the collection of $d\times d$ PLM.
\end{definition}
\begin{remark}
    Notice that the leg submatrices of the $d\times d(2d-1)$ matrix associated with an edge $d$-partition of $K_{2d}$ from \cite{S2Up} are all PLM.
\end{remark}
\begin{remark}
    PLM are called Permutation-Like Matrices because they are share some of the properties of Permutation Matrices but not all of them (every Permutation Matrix is a PLM). One example of a property that holds for Permutation Matrices but not PLM is the sum of each row. Indeed, the sum of the entries in each row of a Permutation Matrix is 1, which is not true for every PLM. For example, the matrix $\begin{pmatrix}
        1&1\\ 0&0
    \end{pmatrix}$ is a PLM, but not a Permutation Matrix. In this sense, PLM are a generalization of Permutation Matrices.
\end{remark}
\begin{remark}
    Notice that every PLM is a Left Stochastic Matrix. This will be further explored in a later section.
\end{remark}
\begin{remark}
    The Stochastic Group of \cite{StochGrp} differs from $PL_d$ since the elements of $PL_d$ are not necessarily invertible. It will become evident in a later section that the only PLM that are in the Stochastic Group are the Permutation Matrices.
\end{remark}
In further studying $PL_d$, it would be helpful to have some notion of the structure of $PL_d$, which is given by the following proposition.
\begin{proposition}\label{PLAlg}
    $PL_d$ is a closed under matrix multiplication.
\end{proposition}
\begin{proof}
    Let $A,B\in PL_d$ for $d\geq1$ and let $C=A\cdot B$. Then,
    \[c_{i,j}=\sum_{p=1}^da_{i,p}b_{p,j}.\]
    However, since $B$ is a PLM, there exists a unique $p_j$ such that $b_{p_j,j}=1$ and $b_{i,j}=0$ for all $i\neq p_j$. Thus, $c_{i,j}=a_{i,p_j}$. 

    Now, since $A$ is a PLM, we further know that given $1\leq j\leq d$, there exists a unique $1\leq i\leq d$ such that $a_{i,j}=1$. In particular, for each $1\leq j\leq d$, there exists a unique $1\leq i\leq d$ such that $a_{i,p_j}=1$. Thus, for each $1\leq j\leq d$, there exists a unique $1\leq i\leq d$ such that $c_{i,j}=a_{i,p_j}=1$. Hence, $C$ is a PLM.
\end{proof}
\begin{remark}
    There are non-invertible PLM, so $PL_d$ is not a subgroup of the General Linear Group. We will discuss more on the invertibility of PLM in a later section. However, it is important to note that there is an identity PLM (the identity matrix) and so $PL_d$ is a monoid.
\end{remark}
\subsection{Symmetric Group Action on Permutation-Like Matrices}

Before we focus on $PL_d$, there are two (left) group actions we would like to focus on, which will simplify some later proofs.

Let $S_d$ be the symmetric group on $d$ letters. There are two group actions on $m\times n$ matrices, one by $S_m$ and one by $S_n$. 

Let $A$ be a $m\times n$ matrix. For $\sigma\in S_m$, we define $\sigma\ast A$ to be the matrix given by permuting the rows of $A$ with $\sigma$. Second, for $\tau\in S_n$, we define $\tau\star A$ to be the matrix given by permuting the columns of $A$ by $\tau^{-1}$. This definition leads to the following lemma.
\begin{lemma}\label{SymGroupLemma}
    Let $A$ be a $m\times n$ matrix, $B$ be a $n\times p$ matrix, $\sigma\in S_m$ and $\tau\in S_p$. Then,
    \begin{enumerate}
        \item $\sigma\ast(A\cdot B)=(\sigma\ast A)\cdot B$.
        \item $\tau\star(A\cdot B)=A\cdot (\tau\star B)$.
    \end{enumerate}
\end{lemma}
Notice that Lemma \ref{SymGroupLemma} immediately follows from basic linear algebra (see \cite{Linear} for example) with the identification of $S_d$ with Permutation Matrices. In particular, row permutation is given by left multiplication of some Permutation Matrix and column permutation is given by right multiplication of some Permutation Matrix. These group actions will be important moving forward and Lemma \ref{SymGroupLemma} will be implicitly used.
\section{Multiplication of Permutation-Like Matrices}
In this section, we give explicit computations for multiplication of PLMs. We will cover these computations using the cases of $d=2$, $d=3$, and $d>3$.
\subsection{The Case $d=2$}
Suppose that $d=2$ and consider $PL_2$. We would like to give multiplication of all of these matrices. Besides the identity matrix $I$, there are three other PLM.
\begin{enumerate}
    \item $R_1^{(2)}=\displaystyle\begin{pmatrix}
        1&1\\
        0&0
    \end{pmatrix}$
    \item $R_2^{(2)}=\displaystyle\begin{pmatrix}
        0&0\\
        1&1
    \end{pmatrix}$
    \item $P_2=\displaystyle\begin{pmatrix}
        0&1\\
        1&0
    \end{pmatrix}$.
\end{enumerate}
\begin{remark}
    We will examine the multiplication closer, but $P_2$ is the matrix whose left multiplication is identified with matrix transposition.
\end{remark}
First, we present the following proposition.
\begin{proposition}\label{RowsDim2}
    Let $m=1$ or $m=2$ and $A\in PL_2$. Then, $R_m^{(2)}\cdot A=R_m^{(2)}$.
\end{proposition}
\begin{proof}
    This proof follows by direct computation, but we give an alternative proof that is illustrative for later sections.
    
    Let $B=R_m^{(2)}\cdot A$. Notice that
    \[b_{i,j}=r_{i,1}b_{1,j}+r_{i,2}b_{2,j}.\]
    However, if $i\neq m$, we know that $r_{i,p}=0$ for $p=1$ and $p=2$. Thus, $b_{i,j}=0$ for all $i\neq m$. Further, we know that there is at least one 1 in each column, so $b_{m,j}=1$ for $j=1$ and $j=2$. Thus,
    \[b_{i,j}=\begin{cases}
        0&\quad i\neq m\\
        1&\quad i=m
    \end{cases}\]
    and hence $B=R_m^{(2)}$.
\end{proof}
Next, we cover the remaining case, which fully classifies multiplication in the case $d=2$.
\begin{proposition}\label{NilpotentDim2}
    Let $d=2$. Then,
    \begin{enumerate}
        \item $P_2\cdot R_1^{(2)}=R_2^{(2)}$.
        \item $P_2\cdot R_2^{(2)}=R_1^{(2)}$.
        \item $P_2\cdot P_2 = I$.
    \end{enumerate}
\end{proposition}
Notice that the proof for Proposition \ref{NilpotentDim2} follows computationally.
\subsection{The Case $d=3$}
Next, we will consider the case of $d=3$. First, notice that in $PL_3$ there are 27 matrices, with one of them being the identity matrix. We start with a special class of PLM called the row PLMs.
\begin{definition}
    For $1\leq m\leq 3$, let $R_m^{(3)}$ be the matrix where $r_{i,j}=0$ if $i\neq m$ and $r_{i,j}=1$ if $i=m$ for all $1\leq j\leq d$. We call $R_m^{(3)}$ the Row Permutation-Like Matrices (row PLM).
\end{definition}
We would like to consider $A\cdot B$, where $A,B$ are PLM. We will start with the case where $A$ is a row PLM.
\begin{proposition}\label{RowDim3}
    Take $1\leq m\leq 3$ and $A\in PL_3$. Then, $R_m^{(3)}\cdot A=R_m^{(3)}$.
\end{proposition}
\begin{proof}
    This proof follows similarly to Proposition \ref{RowsDim2}.

     Let $B=R_m^{(3)}\cdot A$. Notice that
    \[b_{i,j}=r_{i,1}b_{1,j}+r_{i,2}b_{2,j}+r_{i,3}b_{3,j}.\]
    However, if $i\neq m$, we know that $r_{i,p}=0$ for $1\leq p\leq 3$. Thus, $b_{i,j}=0$ for all $i\neq m$. Further, we know that there is at least one 1 in each column, so $b_{m,j}=1$ for $1\leq j\leq 3$. Thus,
    \[b_{i,j}=\begin{cases}
        0&\quad i\neq m\\
        1&\quad i=m
    \end{cases}\]
    and hence $B=R_m^{(3)}$.
\end{proof}
Next, we will consider the case of $A\cdot B$ where $B$ is a row PLM.
\begin{proposition}\label{CPLMwRowDim3}
    Let $A$ be a PLM and let $1\leq m\leq 3$. Then,
    \[A\cdot R_m^{(3)}=\begin{pmatrix}
        a_{1,m}&a_{1,m}&a_{i,m}\\
        a_{2,m}&a_{2,m}&a_{2,m}\\
        a_{3,m}&a_{3,m}&a_{3,m}
    \end{pmatrix}\]
\end{proposition}
\begin{proof}
    Let $B=A\cdot R_m^{(3)}$. Recall that
    \[r_{i,j}=\begin{cases}
        1, & \ \ i=m\\
        0, & \ \ \textrm{otherwise}
    \end{cases}\]
    Therefore, 
    \[b_{i,j}=a_{i,1}r_{1,j}+a_{i,2}r_{2,j}+a_{i,3}r_{3,j}=a_{i,m}r_{m,j}=a_{i,m}.\]
\end{proof}

We have exhausted the cases of $A\cdot B$, where either $A$ or $B$ is a row PLM. Next, we would like to consider several other cases, which all focus on a different type of PLM.
\begin{definition}
    We say that $A\in PL_3$ is a canonical Permutation-Like Matrix (CPLM) if it is a block matrix of the form
    \[A=\begin{pmatrix}
        \boxed{a_{1,1}}&\boxed{0 \ 0}\\
        \boxed{v}&\boxed{A_1}
    \end{pmatrix}\]
    where $v\in k^2$ and $A_1\in PL_2$. We will call $A_1$ the Permutation-Like Component (PLC) of $A$. If $a_{1,1}=1$, we say that $A$ is a leading CPLM.
\end{definition}
With this definition, we have the following lemma, which will simplify the number of cases.
\begin{lemma}\label{CanonicalLemma}
    Let $A\in PL_3$. Then, there exists $\sigma\in S_3$ and CPLM $B$ such that $\sigma\ast A=B$. In other words, any PLM is a CPLM up to swapping rows.
\end{lemma}
\begin{proof}
    First, if $A$ is a row PLM, then $A$ is either a CPLM with a leading element of zero or $(1,2)\ast A$ is a CPLM with a leading element of zero. Indeed, $R_m^{(3)}$ is a CPLM for $m>1$ with PLC given by $R_{m-1}^{(2)}$. Further, in the case of $R_1^{(3)}$, it follows that $(1,2)\ast R_1^{(3)}=R_2^{(3)}$. Next, suppose $A$ is not a row PLM.
    
    If we consider the $d\times d-1$ matrix $\Tilde{A}$ given by $\Tilde{a}_{i,j}=a_{i,j+1}$ for $1\leq i\leq d$ and $1\leq j\leq d-1$, we have that at least one row will be all zeros, as there is at most 1 nonzero entry in each column. Let $r'$ be the row of $\Tilde{A}$ with all zeros and take $\sigma=(1,r')\in S_d$. So, $B=\sigma\ast A$ is the matrix given by $A$ with rows $1$ and $r'$ swapped. Notice that the only column that could have a nonzero element in the first row of $B$ is the first column by the construction of $r'$. Thus, $A$ is a CPLM with PLC given by $\sigma\ast\Tilde{A}$ with the first row removed. 
\end{proof}
With Lemma \ref{CanonicalLemma} in mind, we will be able to reduce our cases down to $A\cdot B$ where $A$ is a CPLM and $B$ is a PLM. First, we will consider the cases where $B$ is also a CPLM. We will start with the case where $B$ is not a leading CPLM.

\begin{proposition}\label{2CPLM1}
    Let $A,B$ be CPLM with PLC $A_1$ and $B_1$, respectively. Suppose that $b_{x,1}=1$ for $x>1$. Then, $A\cdot B$ is given by the block matrix
    \[A\cdot B=\begin{pmatrix}
0&\begin{matrix}0&0\end{matrix}\\
\begin{matrix}a_{2,x}\\a_{3,x}\end{matrix}&A_1\cdot B_1
\end{pmatrix}.\]
\end{proposition}
\begin{proof}
    First, let $A\cdot B=C$ and take the $i^{th}$ row and $j^{th}$ column of $A$, $B$ and $C$ to be denoted by $a_{i,j}$, $b_{i,j}$ and $c_{i,j}$, respectively. We will compute $c_{i,j}$ for all $1\leq i,j\leq 3$.

    Since $A$ is a CPLM, we know that $a_{1,2}=0$ and $a_{1,3}=0$. In particular, we know that
    \[c_{1,j}=a_{1,1}b_{1,j}+a_{1,2}b_{2,j}+a_{1,3}b_{3,j}=a_{1,1}b_{1,j}\]
    for all $1\leq j\leq 3$. However, notice that since $B$ is a CPLM so $b_{1,j}=0$ for $j>1$. Further, $b_{x,1}=1$ for some $x>1$, so we know that $b_{1,1}=0$, as there is at most one 1 in each column. Thus, $c_{1,j}=0$ for all $1\leq j\leq 3$.

    Next, consider $c_{i,1}$ for $i=2$ and $i=3$. We know that $a_{1,3}=0$, $b_{x,1}=1$, and $b_{i,1}=0$ for $i\neq x$. Therefore,
    \[c_{i,1}=a_{i,1}b_{1,1}+a_{i,2}b_{2,1}+a_{i,3}b_{3,1}=a_{i,x}b_{x,1}=a_{i,x}.\]
    In particular, $c_{2,1}=a_{2,x}$ and $c_{3,1}=a_{3,x}$.

    Lastly, since $a_{2,1}=a_{3,1}=0$, it follows for $2\leq i,j\leq 3$ that
    \[c_{i,j}=a_{i,1}b_{1,j}+a_{i,2}b_{2,j}+a_{i,3}b_{3,j}=a_{i,2}b_{2,j}+a_{i,3}b_{3,j}.\]
    Notice that $\begin{pmatrix}
        c_{2,2}&c_{2,3}\\
        c_{3,2}&c_{3,3}
    \end{pmatrix}$ is precisely $A_1\cdot B_1$ since
    \[A_1=\begin{pmatrix}
        a_{2,2}&a_{2,3}\\
        a_{3,2}&a_{3,3}
    \end{pmatrix}\]
    and
    \[B_1=\begin{pmatrix}
        b_{2,2}&b_{2,3}\\
        b_{3,2}&b_{3,3}
    \end{pmatrix}.\]
    Therefore, $C$ is of the form $\begin{pmatrix}
            0&0 \ 0\\
            \begin{matrix}
                a_{2,x}\\
                a_{3,x}
            \end{matrix}&A_1\cdot B_1
        \end{pmatrix}.$
\end{proof}
Next, we consider the case where $B$ is a leading PLM.
\begin{proposition}\label{2CPLM2}
    Let $A,B$ be CPLM with PLC $A_1$ and $B_1$, respectively. If $B$ is a leading CPLM, then $C=A\cdot B$ is a CPLM with PLC $A_1\cdot B_1$. Further, $c_{i,1}=a_{i,1}$ for all $1\leq i\leq 3$.
\end{proposition}
\begin{proof}
    First, notice that since $B$ is a leading CPLM, we know that $b_{1,1}=1$ and $b_{2,1}=b_{3,1}=0$. 

    Next, let's compute the first column of $C$. Since $b_{1,1}=1$ and $b_{2,1}=b_{3,1}=0$, we know
    \[c_{1,1}=a_{1,1}b_{1,1}+a_{1,2}b_{2,1}+a_{1,3}b_{3,1}=a_{1,1},\]
    \[c_{2,1}=a_{2,1}b_{1,1}+a_{2,2}b_{2,1}+a_{1,3}b_{3,1}=a_{2,1},\]
    and
    \[c_{3,1}=a_{3,1}b_{1,1}+a_{3,2}b_{2,1}+a_{3,3}b_{3,1}=a_{3,1}.\]
    In particular, we have that for all $1\leq i\leq 3$ that $c_{i,1}=a_{i,1}$.

    Next, consider the other entries in the first row of $C$. This is given by
    \[c_{1,2}=a_{1,1}b_{1,2}+a_{1,2}b_{2,2}+a_{1,3}b_{3,2}=0\]
    and
    \[c_{1,3}=a_{1,1}b_{1,3}+a_{1,2}b_{2,3}+a_{1,3}b_{3,3}=0\]
    since $a_{1,2}=a_{1,3}=b_{1,2}=b_{1,3}=0$.

    Lastly, consider $c_{i,j}$ for $2\leq i,j\leq 3$. Since $b_{1,j}=0$ for $2\leq j\leq 3$, it follows that
    \[c_{i,j}=a_{i,1}b_{1,j}+a_{i,2}b_{2,j}+a_{i,3}b_{3,j}=a_{i,2}b_{2,j}+a_{i,3}b_{3,j}.\]
    Further, notice that 
    \[A_1=\begin{pmatrix}
        a_{2,2}&a_{2,3}\\
        a_{3,2}&a_{3,3}
    \end{pmatrix}\]
    and
    \[B_1=\begin{pmatrix}
        b_{2,2}&b_{2,3}\\
        b_{3,2}&b_{3,3}
    \end{pmatrix}\]
    so the entries making up the block $c_{i,j}$ where $2\leq i,j\leq3$ give $A_1\cdot B_1$ precisely. Therefore, $C$ is a PLM with PLC $A_1\cdot B_1$.
\end{proof}
Note Propositions \ref{2CPLM1} and \ref{2CPLM2} completely determine multiplication of two CPLM matrices. In particular, we have the following immediate corollary.
\begin{corollary}
    $3\times 3$ Canonical Permutation-Like Matrices are closed under matrix multiplication.
\end{corollary}

Before we consider some other cases of matrix multiplication, we want to give the following special case that follows as a corollary to Proposition \ref{2CPLM1}.
\begin{corollary}\label{GenRowsDim3}
    Let $A,B$ be CPLM with PLC $A_1$ and $B_1$, respectively. Further, suppose that $A_1=R_m^{(2)}$ for $1\leq m\leq 2$ and that $B$ is not a leading CPLM. Then, $A\cdot B=R_{m+1}^{(3)}$.
\end{corollary}
\begin{remark}
    Notice as a result of Corollary \ref{GenRowsDim3} that if $A$ is a CPLM with leading element zero and PLC $R_m^{(2)}$, then $A^2=R_{m+1}^{(3)}$. This computation will be helpful for approaching the problem of periodicity in Section 4.
\end{remark}
Next, we want to consider some other cases of matrix multiplication. We start with a definition.
\begin{definition}
    We say that $A\in PL_3$ is a pre-canonical Permutation-Like Matrix (PCPLM) if there exists a column permutation $\tau\in S_3$ such that $\tau\star A$ is a CPLM. 
\end{definition}
\begin{remark}
    Notice that every PLM with zero 1s or only one 1 in the first row is either a CPLM or a PCPLM.
\end{remark}
Now, the following comes immediately from the definition with Lemma \ref{SymGroupLemma}.
\begin{proposition}\label{PCPLM}
    Let $A,B\in PL_3$ such that $A$ is a CPLM and $B$ is a PCPLM such that $\tau\star B$ is a CPLM for some $\tau\in S_d$. Then,
    \[A\cdot B=\tau^{-1}\star(A\cdot (\tau\star B))\]
\end{proposition}
In particular, using Propositions \ref{2CPLM1} and \ref{2CPLM2}, we have explicit formulas that arise from Proposition \ref{PCPLM}. 

The last class of PLM we will focus on is related to when we are not able to get a PLM in the form of a CPLM.
\begin{definition}
    We say that $A\in PL_3$ is an irregular Permutation-Like Matrix (IPLM) if it is not a CPLM, PCPLM, or row PLM. 
\end{definition}
This definition leads into the following proposition.
\begin{proposition}\label{Irregular}
    Let $A,B\in PL_3$ such that $A$ is a CPLM with PLC $A_1$ and $B$ is an IPLM. Then, we have the following cases for their multiplication:
    \begin{enumerate}
        \item $\displaystyle\begin{pmatrix}
            a_{1,1}&0&0\\
            a_{2,1}&a_{2,2}&a_{2,3}\\
            a_{3,1}&a_{3,2}&a_{3,3}
            \end{pmatrix}\cdot\begin{pmatrix}
            1&1&0\\
            0&0&b\\
            0&0&c
        \end{pmatrix}=\begin{pmatrix}
            1&1&0\\
            0&0&a_{2,2}b+a_{2,3}c\\
            0&0&a_{3,2}b+a_{3,3}c
        \end{pmatrix}$
        \item $\displaystyle\begin{pmatrix}
            a_{1,1}&0&0\\
            a_{2,1}&a_{2,2}&a_{2,3}\\
            a_{3,1}&a_{3,2}&a_{3,3}
            \end{pmatrix}\cdot\begin{pmatrix}
            1&0&1\\
            0&b&0\\
            0&c&0
        \end{pmatrix}=\begin{pmatrix}
            1&0&1\\
            0&a_{2,2}b+a_{2,3}c&0\\
            0&a_{3,2}b+a_{3,3}c&0
        \end{pmatrix}$
        \item $\displaystyle\begin{pmatrix}
            a_{1,1}&0&0\\
            a_{2,1}&a_{2,2}&a_{2,3}\\
            a_{3,1}&a_{3,2}&a_{3,3}
            \end{pmatrix}\cdot\begin{pmatrix}
            0&1&1\\
            b&0&0\\
            c&0&0
        \end{pmatrix}=\begin{pmatrix}
            0&1&1\\
            a_{2,2}b+a_{2,3}c&0&0\\
            a_{3,2}b+a_{3,3}c&0&0
        \end{pmatrix}$
    \end{enumerate}
\end{proposition}
\begin{proof}
    Let $B$ be an IPLM. If there are three 1s in the first row, it is $R_1^{(3)}$ (and hence not a IPLM). So, $B$ has zero, one, or two 1s in the first row. 
    
    If $B$ has zero 1s in the first row, then it is of the form
    \[\begin{pmatrix}
        0&0&0\\
        \ast&\ast&\ast\\
        \ast&\ast&\ast
    \end{pmatrix}\]
    Since there is exactly one 1 in each column, that means that the 1 for the second and third column appears in the second or third row. This means that we have
    \[\begin{pmatrix}
        0&\begin{matrix}0&0\end{matrix}\\
        \begin{matrix}\ast\\\ast\end{matrix}&B_1
    \end{pmatrix}\]
    where $B_1\in PL_2$. Therefore, $B$ would have to be a CPLM (which contradicts that it is an IPLM). 
    
    Similarly, we can suppose that $B$ has one 1 in the first row, so up to column permutation it is of the form
    \[\begin{pmatrix}
        1&0&0\\
        0&\ast&\ast\\
        0&\ast&\ast
    \end{pmatrix}\]
    so we know there are two 0s and two 1s the fill in the remaining positions (one of each in each column) and hence $B$, up to column permutation, is a CPLM (i.e. $B$ is a CPLM or PCPLM, which contradicts that it is an IPLM). Therefore, if $B$ is an IPLM, there is exactly two 1s in the first row. These are covered in the three cases.
    
    Given this, case 1 results by direct computation and cases 2 and 3 result from a column permutation of case 1. 
\end{proof}

Lastly, the claim is that we have covered all of the possible cases.
\begin{theorem}\label{PL3Thm}
    Let $A,B\in PL_3$. Then, $A\cdot B$ is given, up to row permutation, by Propositions \ref{RowDim3}-\ref{Irregular}.
\end{theorem}
\begin{proof}
    Notice that if $A$ is not a row PLM (a case covered by Proposition \ref{RowDim3}), we can assume that $A$ is a CPLM without loss of generality by Lemma \ref{CanonicalLemma}. By the definition of IPLM, we know that every PLM is a row PLM, CPLM, PCPLM, or IPLM. All of these cases are covered by Propositions \ref{CPLMwRowDim3}-\ref{Irregular}.
\end{proof}

\subsection{The Case $d>3$}
Next, we will consider the case $d>3$. This section will follow in parallel with the case $d=3$.

\begin{definition}
    Let $R_m^{(d)}$ be the $d\times d$ matrix with entry $r_{i,j}$ in row $i$ and column $j$, where $r_{i,j}=1$ if $i=m$ and $r_{i,j}=0$ otherwise. We call $R^{(d)}_m$ a row Permutation-Like Matrix (row PLM).
\end{definition}
Related to this special matrix, we have the following proposition in parallel with Proposition \ref{RowDim3}.
\begin{proposition}\label{rowPLMGen}
    Let $d\geq1$ and $1\leq m\leq d$. If $A\in PL_d$ then $R_m^{(d)}\cdot A=R_m^{(d)}$.
\end{proposition}
\begin{proof}
    The proof of Proposition \ref{rowPLMGen} follows similarly to Propositions \ref{RowsDim2} and \ref{RowDim3}. 

    Let $B=R_m^{(d)}\cdot A$. Then, since $r_{i,j}=1$ if $i=m$ and $r_{i,j}=0$ otherwise, if $i\neq m$ we have
    \[b_{i,j}=\sum_{p=1}^dr_{i,p}a_{p,j}=0.\]
    Further, since PLM are closed under multiplication, we know $B$ is a PLM. Thus, $B$ has exactly one 1 in each column. Hence, because $b_{i,j}=0$ for all $i\neq m$, we have that $b_{m,j}=1$ for all $1\leq j\leq d$.
\end{proof}

Similar to the case $d=3$, we can also consider matrix multiplication of $A\cdot R_m^{(d)}$, where $1\leq m\leq d$..

\begin{proposition}\label{CPLMwRowLD}
    Let $A$ be a PLM and let $1\leq m\leq d$. If $B=A\cdot R^{(d)}_m$, then $b_{i,j}=a_{i,m}$ for all $1\leq i,j\leq d$.
\end{proposition}
\begin{proof}
    This follows a similar proof as Proposition \ref{CPLMwRowDim3}, which we give for completeness.
    
    Let $B=A\cdot R_m^{(d)}$. Recall that
    \[r_{i,j}=\begin{cases}
        1, & \ \ i=m\\
        0, & \ \ \textrm{otherwise}
    \end{cases}\]
    Therefore, for all $1\leq i,j\leq d$,
    \[b_{i,j}=\sum_{p=1}^da_{i,p}r_{p,j}=a_{i,m}r_{m,j}=a_{i,m}.\]
\end{proof}

Since we have considered the case of row PLM, we now define CPLM for $d>3$.
\begin{definition}
    We say that $A\in PL_d$ is a canonical Permutation-Like Matrix (CPLM) if it is a block matrix of the form
    \[A=\begin{pmatrix}
        \boxed{a_{1,1}}&\boxed{0 \ 0 \ \ldots \ 0}\\
        \boxed{\begin{matrix}
            a_{2,1}\\ a_{3,1}\\ \vdots \\ a_{d,1}
        \end{matrix}}&\boxed{A_1}
    \end{pmatrix}\]
    where $A_1\in PL_{d-1}$. We call $A_1$ the permutation-like component (PLC) of $A$. We say $A$ is a leading CPLM if $a_{1,1}=1$.
\end{definition}
\begin{remark}
    Notice that $R^{(d)}_m$ are CPLM for $m>1$.
\end{remark}
Similar to the case $d=3$, we have the following lemma that reduces the number of cases.
\begin{lemma}\label{PLSwapLargeDim}
    Let $A\in PL_d$. Then, there exists $\sigma\in S_d$ and CPLM $B$ such that $\tau\ast A=B$.
\end{lemma}
Lemma \ref{PLSwapLargeDim} follows a similar proof as Lemma \ref{CanonicalLemma}, which we give for completeness.
\begin{proof}
    First, if $A$ is a row PLM, then either it is a CPLM with a leading element of zero or $(1,2)\ast A$ is a CPLM with a leading element of zero. Indeed, $R_m^{(d)}$ is a CPLM for $m>1$ with PLC $R_{m-1}^{(d-1)}$. Further, in the case of $R_1^{(d)}$, it follows that $(1,2)\ast R_1^{(d)}=R_2^{(d)}$. Next, suppose $A$ is not a row PLM.
    
    If we look at the $d\times d-1$ matrix $\Tilde{A}$ given by $\Tilde{a}_{i,j}=a_{i,j+1}$ for $1\leq i\leq d$ and $1\leq j\leq d-1$, we have at least one row will be all zeros, as there is at most 1 nonzero element in each column. Let $r'$ be the row of $\Tilde{A}$ with all zeros and take $\sigma=(1,r')$. So, $\sigma\ast A$ is the matrix given by swapping rows 1 and $r'$ in $A$. Further, by construction, we know that the first row of $\sigma\ast A$ in columns $2,3,\ldots, d$ are all zero. Thus, $\sigma\ast A$ is a CPLM.
\end{proof}
With Lemma \ref{PLSwapLargeDim}, we now reduce the problem down to $A\cdot B$, where $A$ is a CPLM and $B\in PL_d$. In parallel with the case $d=3$, we will start with considering where $B$ is a CPLM. This is broken down into the cases where $B$ is a leading CPLM and $B$ is not a leading CPLM. We start with the latter.
\begin{proposition}\label{CPLM1LargeDim}
    Let $A,B$ be CPLM with PLC $A_1$ and $B_1$, respectively. Then, if the leading element of $B$ is zero and $b_{x,1}=1$,
    \[A\cdot B=\begin{pmatrix}
        0 & 0 \ 0 \ \ldots \ 0\\
        v & A_1\cdot B_1
    \end{pmatrix}\]
    where $v$ is the $(x-1)^{st}$ column of $A_1$. 
\end{proposition}
\begin{proof}
    Let $A\cdot B=C$ and denote the $i^{th}$ row and $j^{th}$ column of $A$, $B$, and $C$ by $a_{i,j}, b_{i,j},$ and $c_{i,j},$ respectively. Notice that the $(x-1)^{st}$ column of $A_1$ is given by the entries $a_{i+1,x}$ for $1\leq i\leq d-1$.

    Now, since $A$ and $B$ are CPLM, we know that $a_{1,j}=b_{1,j}=0$ for all $2\leq j\leq d$. In particular, we know that 
    \[c_{1,j}=\sum_{p=1}^d a_{1,p}b_{p,j}=a_{1,1}b_{1,j}\]
    for all $1\leq j\leq d$. However, if $j>1$, then $c_{1,j}=0$ since $b_{1,j}=0$, as $B$ is a CPLM. Further, since the leading element of $B$ is zero, we know that $c_{1,1}=0$, as $b_{1,1}=0$.

    Next, consider the remaining terms in the first column. By assumption, we took that $b_{x,1}=1$ and $b_{i,1}=0$ for all $i\neq x$. Thus, 
    \[c_{i,1}=\sum_{p=1}^da_{i,p}b_{p,1}=a_{i,x}b_{x,1}=a_{i,x}\]
    for all $2\leq i\leq d$. In particular, notice that $a_{i,x}$ is the $(i-1)^{st}$ row and $(x-1)^{st}$ column of $A_1$. Thus, the vector $v=\begin{pmatrix}
        c_{2,1}\\ c_{3,1} \\ \vdots \\ c_{d,1}
    \end{pmatrix}$ gives the $(x-1)^{st}$ column of $A_1$.

    Lastly, consider the terms $c_{i,j}$, where $2\leq i,j\leq d$. Remember that $b_{1,j}=0$ for all $j\geq2$, so
    \[c_{i,j}=\sum_{p=1}^d a_{i,p}b_{p,j}=\sum_{p=2}^d a_{i,p}b_{p,j}.\]
    However, since $a_{i+1,j+1}$ and $b_{i+1,j+1}$ for $1\leq i,j\leq d-1$ determines $A_1$ and $B_1$, respectively, we know that $c_{i,j}$ is given by the $(i-1)^{st}$ row and $(j-1)^{st}$ column of $A_1\cdot B_1$ for $2\leq i,j\leq d$. Therefore, $C$ is a CPLM and satisfies the statement of the proposition.
\end{proof}
Next, we give the parallel proposition to Proposition \ref{2CPLM2}.
\begin{proposition}\label{CPLM2LargeDim}
    Let $A,B$ be CPLM with PLC $A_1$ and $B_1$, respectively. If $B$ is a leading CPLM, then $C=A\cdot B$ is a CPLM with PLC $A_1\cdot B_1$. Further if $a_{i,j}$ and $c_{i,j}$ denote the $i^{th}$ row and $j^{th}$ column of $A$ and $C$, respectively, then $a_{m,1}=c_{m,1}$ for all $1\leq m\leq d$.
\end{proposition}
\begin{proof}
    First, consider $c_{1,j}$ for $1\leq j\leq d$. Since $B$ is a leading CPLM, we know that $b_{1,1}=1$ and $b_{1,j}=0$ for all $j>1$. Further, since $A$ is a CPLM, we know that $a_{1,j}=0$ for $j>1$. Thus, 
    \[c_{1,j}=\sum_{p=1}^da_{1,p}b_{p,j}=a_{1,1}b_{1,j}.\]
    This gives us that $c_{1,j}=0$ for $j>1$ and $c_{1,1}=a_{1,1}$.
    
    Next, consider $c_{i,1}$ for $1\leq i\leq d$. Since $B$ is a leading CPLM, we know that $b_{i,1}=0$ for all $i>1$ with $b_{1,1}=1$. Thus,
    \[c_{i,1}=\sum_{p=1}^da_{i,p}b_{p,1}=a_{i,1}b_{1,1}.\]
    Hence, $c_{i,1}=a_{i,1}$ for all $1\leq i\leq d$.

    Lastly, consider $c_{i,j}$ for all $1<i,j\leq d$. Notice that
    \[A_1=\begin{pmatrix}
        a_{2,2}&a_{2,3}&\ldots&a_{2,d}\\
        a_{3,2}&a_{3,3}&\ldots&a_{3,d}\\
        \vdots&\vdots&\ddots&\vdots\\
        a_{d,2}&a_{d,3}&\ldots&a_{d,d}
    \end{pmatrix}\]
    and
    \[B_1=\begin{pmatrix}
        b_{2,2}&b_{2,3}&\ldots&b_{2,d}\\
        b_{3,2}&b_{3,3}&\ldots&b_{3,d}\\
        \vdots&\vdots&\ddots&\vdots\\
        b_{d,2}&b_{d,3}&\ldots&b_{d,d}
    \end{pmatrix}.\]
    Now, since $a_{i,1}=b_{1,j}=0$ for all $1<i,j\leq d$, as $A$ is a CPLM and $B$ is a leading CPLM, we know that
    \[c_{i,j}=\sum_{p=1}^da_{i,p}b_{p,j}=\sum_{p=2}^da_{i,p}b_{p,j}.\]
    Notice that the matrix $C_1=\begin{pmatrix}
        c_{2,2}&c_{2,3}&\ldots&c_{2,d}\\
        c_{3,2}&c_{3,3}&\ldots&c_{3,d}\\
        \vdots&\vdots&\ddots&\vdots\\
        c_{d,2}&c_{d,3}&\ldots&c_{d,d}
    \end{pmatrix}$ is precisely $A_1\cdot B_1$. Therefore, $C$ is a CPLM with PLC $C_1=A_1\cdot B_1$.
\end{proof}
Propositions \ref{rowPLMGen} and \ref{CPLM2LargeDim} yield the following corollary.
\begin{corollary}\label{GenRowsBigD}
    Let $A,B$ be CPLM with PLC $A_1$ and $B_1$, respectively. Further, suppose that $A_1=R_m^{(d-1)}$ for $1\leq m\leq d-1$ and that $B$ is not a leading CPLM. Then, $A\cdot B=R_{m+1}^{(d)}$.
\end{corollary}
Now, we have covered the case of CPLM, so we will consider the parallel case to that covered in Proposition \ref{PCPLM}.

\begin{definition}
    Let $A\in PL_d$. Then, we say that $A$ is a pre-Canonical Permutation-Like Matrix (PCPLM) if there exists $\tau\in S_d$ such that $\tau\star A$ is a CPLM. 
\end{definition}

Now, the definition of PCPLM leads immediately to the following.
\begin{proposition}
    Let $A,B\in PL_d$ such that $A$ is a CPLM and $B$ is a PCPLM such that $\tau\star B$ is a CPLM for some $\tau\in S_d$. Then,
    \[A\cdot B=\tau^{-1}\star(A\cdot (\tau\star B)).\]
\end{proposition}

Lastly, we would like to consider the following case.
\begin{definition}
    Let $A\in PL_d$. Then, we say that $A$ is an Irregular Permutation-Like Matrix (IPLM) if it is not a row PLM, CPLM, or PCPLM.
\end{definition}
    For $B\in PL_d$, if $B$ is irregular and we take $\zeta(B)$ to be the number of nonzero entries in the first row of $B$, then this definition of IPLM implies that $\zeta(B)>1$. This follows by a similar proof as the argument in the proof of Proposition \ref{Irregular}, since $B$ is not CPLM or PCPLM. Further, if all of the entries in the first row are 1, then $B$ is a row PLM, hence not an IPLM. So, the only case left to consider is that of $1<\zeta(B)<d$.
\begin{remark}
    We use $\zeta(B)$ because it is related to the number of zero entries in the first row. Indeed, $d-\zeta(B)$ is the number of zero entries in the first row of $B$.
\end{remark}
\begin{proposition}\label{IPLM2}
    Let $X(A,n)$ be the $d-1\times d-1$ matrix given by 
    \[X(A,n)=\begin{pmatrix}
        a_{2,1}&a_{2,1}&\ldots&a_{2,1}&0&0&\ldots&0\\
        a_{3,1}&a_{3,1}&\ldots&a_{3,1}&0&0&\ldots&0\\
        a_{4,1}&a_{4,1}&\ldots&a_{4,1}&0&0&\ldots&0\\
        \vdots&\vdots&\ddots&\vdots&\vdots&\vdots&\ddots&\vdots\\
        a_{d,1}&a_{d,1}&\ldots&a_{d,1}&0&0&\ldots&0
    \end{pmatrix}\]
    where the first n columns of row i have the entries $a_{i+1,1}$.
    
    Suppose that $A,B\in PL_d$ such that $A$ is a CPLM with PLC $A_1$. Further suppose $B$ is an IPLM with $1<\zeta(B)<d$.
    \begin{enumerate}
        \item Suppose $B$ is a block matrix of the form $B=\begin{pmatrix}
            1 \ 1 \ \ldots \ 1 & 0 \ 0 \ \ldots \ 0\\
            \textbf{0} & B_1
        \end{pmatrix}$ where $\textbf{0}$ is the $d\times \zeta(B)$ zero matrix and $B_1$ is a $d-1\times d-\zeta(B)$ matrix. Take $B_2$ to be the $d-1\times d-1$ matrix given by augmenting $B_1$ on the left by the $d-1\times \zeta(B)-1$ zero matrix. Then $A\cdot B$ is given by the block matrix
        \[A\cdot B=\begin{pmatrix}
            \begin{matrix}a_{1,1} & a_{1,1} & \ldots & a_{1,1}\end{matrix} & \begin{matrix}0 & 0 & \ldots & 0\end{matrix}\\
            &\\
            \begin{matrix}v & v & \ldots & v\end{matrix} & X(A,\zeta(B)-1) + A_1\cdot B_2
        \end{pmatrix}\]
        where $v=\begin{pmatrix}
            a_{2,1}\\ a_{3,1}\\ \vdots \\ a_{d,1}
        \end{pmatrix}$.
        \item Otherwise, there exists a $\tau\in S_d$ such that $\Tilde{B}=\tau\star B$ is of the form of the previous case and
        \[A\cdot B=\tau^{-1}\star\begin{pmatrix}
            \begin{matrix}a_{1,1} & a_{1,1} & \ldots & a_{1,1}\end{matrix} & \begin{matrix}0 & 0 & \ldots & 0\end{matrix}\\
            &\\
            \begin{matrix}v & v & \ldots & v\end{matrix} & X(A,\zeta(\Tilde{B})-1) + A_1\cdot \Tilde{B}_2
        \end{pmatrix}\]
        where $v=\begin{pmatrix}
            a_{2,1}\\ a_{3,1}\\ \vdots \\ a_{d,1}
        \end{pmatrix}$.
    \end{enumerate}
\end{proposition}
\begin{proof}
    The latter statement is immediate since for all IPLM $B$, we know $1<\zeta(B)<d$ and hence we can put $B$ in the form of the former case, up to column permutations. So, it suffices to prove the former.

    Suppose that $B=\begin{pmatrix}
            1 \ 1 \ \ldots \ 1 & 0 \ 0 \ \ldots \ 0\\
            \textbf{0} & B_1
        \end{pmatrix}$ and let $C=A\cdot B$. Take $B_2$ to be the $d-1\times d-1$ matrix given by augmenting $B_1$ on the left by the zero matrix. We will compute $c_{i,j}$ for $1\leq i,j\leq d$.

    First, consider $c_{i,1}$ for all $1\leq i\leq d$. Now, there exists some $1\leq m\leq d$ such that $a_{m,1}=1$ and $a_{i,1}=0$ for $i\neq m$. Further, we know that $b_{1,1}=1$ and $b_{i,1}=0$ for $i>1$. Thus,
    \[c_{i,1}=\sum_{p=1}^da_{i,p}b_{p,1}=a_{i,1}.\]
    So, $c_{m,1}=1$ and $c_{i,1}=0$ for $i\neq m$.

    Next, consider $c_{1,j}$ for all $1<j\leq d$. Since $A$ is a CPLM, we know that $a_{1,j}=0$ for $j>1$. So,

    \[c_{1,j}=\sum_{p=1}^da_{1,p}b_{p,j}=a_{1,1} b_{1,j}\]
    However, there exists $N>1$ such that $b_{1,j}=0$ if $j>N$ and $b_{1,j}=1$ if $j\leq N$. Thus, $c_{1,j}=0$ if $j>N$ and $c_{1,j}=a_{1,1}$ if $j\leq N$.

    Now we will consider $c_{i,j}$ for all $1<i,j\leq d$. Since $b_{1,j}=0$ for $j>N$, we know that
    \[c_{i,j}=\sum_{p=1}^da_{1,p}b_{p,j}=\sum_{p=2}^d a_{i,p}b_{p,j}\]
    for $j>N$. Also, since $b_{1,j}=1$ for $j\leq N$, we know that
    \[c_{i,j}=a_{i,1}+\sum_{k=2}^da_{i,k}b_{k,j}.\]
    In particular, notice that, up to addition by an element of $A$ in the first $\zeta(B)-1$ columns, the matrix $C_1=\begin{pmatrix}
        c_{2,2}&c_{2,3}&\ldots&c_{2,d}\\
        c_{3,2}&c_{3,3}&\ldots&c_{3,d}\\
        \vdots&\vdots&\ddots&\vdots\\
        c_{d,2}&c_{d,3}&\ldots&c_{d,d}
    \end{pmatrix}$ corresponds to $A_1\cdot B_2$. Further, this element of $A$ is explicitly given in $X(A,\zeta(B)-1)$ and since addition is pointwise, $C_1=X(A,\zeta(B)-1)+A_1\cdot B_2$. Therefore, we have computed $C$ explicitly and the statement holds.
\end{proof}
Proposition \ref{IPLM2} covers our final case, culminating in the following theorem.
\begin{theorem}
    Let $A,B\in PL_d$ for $d>3$. Then, $A\cdot B$ is given, up to row permutation, by Propositions \ref{rowPLMGen}-\ref{IPLM2}.
\end{theorem}
\begin{proof}
    This follows similarly to the case of $d=3$. If $A$ is a row PLM, we can apply Proposition \ref{rowPLMGen}. Otherwise, we can assume that $A$ is a CPLM by Lemma \ref{PLSwapLargeDim}. By the definition of IPLM, we know that every PLM is a row PLM, CPLM, PCPLM, or IPLM. All of these cases are covered by Propositions \ref{CPLM1LargeDim}-\ref{IPLM2}.
\end{proof}

\section{Properties of Permutation-Like Matrices}
In this section, we will give some properties of PLM, focusing on $A\in PL_d$ for $d=2$ or $d=3$. Throughout this section, assume that $k$ has characteristic 0.

\subsection{Periodicity of Permutation-Like Matrices}
First, we will look at the periodicity of PLM. Recall that a periodic matrix $A$ is a square matrix such that $A^{k+1}=A$. The minimal such $k$ is called the index of $A$ \cite{Ayres}. 
\begin{proposition}\label{Period}
    Let $A\in PL_d$ for $d=2$ or $d=3$. Then, $A$ is a periodic matrix or there exists $k\geq1$ such that $A^2$ is a row PLM.
\end{proposition}

Proposition \ref{Period} can be verified computationally, where either $A^2$ is a row PLM or $A^2$, $A^3$, or $A^4$ is $A$. Further, since row PLM absorb right multiplication in $PL_d$, we know have a distinction here between these two classes of PLM: these ones that are periodic and the ones that have $A^k$ are row PLM for sufficiently large $k$. With this in  mind, we have the following definition.
\begin{definition}
    Let $A\in PL_d$ for $d>1$. If there exists $k\geq1$ such that $A^{k+1}$ is a row PLM, we say that $A$ is a pre-row Permutation-Like Matrix (pre-row PLM). 
\end{definition}
\begin{remark}
    Notice that every row PLM is a pre-row PLM.
\end{remark}
\begin{remark}
    If $A$ is pre-row and not a row PLM, then $A$ is not periodic. This arises due to Proposition \ref{rowPLMGen}. For example,
    \[A=\begin{pmatrix}
        0&0&0\\
        1&0&0\\
        0&1&1
    \end{pmatrix}\]
    is pre-row since
    \[A^2=\begin{pmatrix}
        0&0&0\\
        0&0&0\\
        1&1&1
    \end{pmatrix}\]
    is a row PLM. Further, $A^m=A^2\cdot A^{m-2}=R_3^{(3)}\cdot A^{m-2}=R_3^{(3)}$ for all $m>2$. In particular, there is no $m$ such that $A^{m+1}=A$.
\end{remark}
\begin{remark}
    If we consider the case $d=3$, note that if $A$ is a CPLM that is not a leading CPLM and has a row PLM $R_m^{(d-1)}$ for its PLC, then $A^2=R_{m+1}^{(d)}$. This extends for $d>3$, as stated in Corollary \ref{GenRowsBigD}, but it is not immediate that these are the only pre-row PLM.
\end{remark}
\subsection{Eigenvalues of PLM}
Next, we would like to consider the eigenvalues of PLM. As a result of Proposition \ref{Period}, we have the following corollary related to eigenvalues.
\begin{corollary}
    Let $A\in PL_d$ be a PLM for $d=2$ or $d=3$. Then every eigenvalues of $A$ is either 0 or a root of unity.
\end{corollary}
\begin{proof}
   By Proposition \ref{Period}, we know that every $PL_d$ for $d=2$ or $d=3$ is a pre-row PLM or periodic.
   
   Suppose $A$ is periodic. Then, there exists $m$ such that $A^{m+1}-A=0$. In particular, we know that the eigenvalues of $A$ are roots of the polynomial $y=x^{m+1}-x$. The only possibilities are $m^{th}$ roots of unity and 0.

   Now, suppose $A$ is a pre-row PLM. Then, there exists $m$ such that $A^m$ is a row PLM. Further, for every row PLM $R^{(d)}_n$, it follows by Proposition \ref{rowPLMGen} that $\left(R^{(d)}_n\right)^2=R^{(d)}_n$. In particular, $A^{2m}-A^m=0$, so the eigenvalues of $A$ are roots of the polynomial $y=x^{2m}-x^m$. The only possibilities are $m^{th}$ roots of unity and 0.
\end{proof}

We do not currently have a result concerning the eigenvalues of PLM for $d>3$, except for certain cases. For pre-row PLM, we do have the following result that comes from the definition of pre-row PLM.
\begin{proposition}
    Let $A$ be a $d\times d$ pre-row PLM for $d>1$. If $\lambda$ is an eigenvalue of $A$, then $\lambda=0$ or $\lambda$ is a roots of unity.
\end{proposition}
\begin{proof}
    Since $A$ is a pre-row PLM there exists $k$ such that $A^k$ is a row PLM. Further, for every row PLM $R^{(d)}_m$, it follows by Proposition \ref{rowPLMGen} that $\left(R^{(d)}_m\right)^2=R^{(d)}_m$. In particular, $A^{2k}-A^k=0$, so the eigenvalues of $A$ are roots of the polynomial $y=x^{2k}-x^k$. The only possibilities are $k^{th}$ roots of unity and 0.
\end{proof}
This leads to the following corollary.
\begin{corollary}\label{LDimEigenSpecial}
    Let $A$ be a $d\times d$ CPLM with PLC $A_1$ for $d>1$. Further suppose $A_1$ is a row PLM and $A$ is not a leading PLM. Then, the eigenvalues of $A$ are zero or roots of unity.
\end{corollary}
Notice that Corollary \ref{LDimEigenSpecial} follows immediately from previous discussion that CPLM with leading element zero and PLC given by a row PLM are pre-row PLM.

Instead of giving a concrete answer to the eigenvalues of PLM in general, we will give a partial answer by computing the spectral radius of such matrices.

Recall that given complex eigenvalues $\lambda_1,\lambda_2,\ldots,\lambda_d$ for a $d\times d$ matrix $A$, the spectral radius is given by $\rho(A)=max\{|\lambda_1|,|\lambda_2|,\ldots,|\lambda_d|\}$, where $|z|$ is the usual norm on the complex numbers. Further, recall given a matrix norm $||\cdot ||$, we know that $\rho(A)\leq ||A||$. This leads to the following proposition.
\begin{proposition}\label{SpectralProp}
    Let $A\in PL_d(\mathbb{C})$ for all $d\geq1$. Then, $\rho(A)\leq1$.
\end{proposition}
\begin{proof}
    This easily follows using the 1-norm. Recall that the 1-norm of $A$ to be given by
    \[||A||_1=max_{1\leq j\leq d}\sum_{i=1}^da_{i,j}.\]
    However, since $A$ is a PLM, for each $1\leq j\leq d$, there exists a unique $1\leq m\leq d$ such that $a_{m,j}=1$, with all other entries in that column zero. Thus, $||A||_1=1$ and hence $\rho(A)\leq 1$.
\end{proof}

As a last note on eigenvalues, we have the following proposition and the related corollary.
\begin{proposition}\label{NonPermZero}
    Let $A\in PL_d$ for $d>1$. If $A$ is not a Permutation Matrix, then $0$ is an eigenvalue for $A$.
\end{proposition}

\begin{proof}
    Suppose that $A$ is not a Permutation Matrix. Notice that since $A$ is PLM, it has exactly $d$ nonzero entries, which are all equal to 1 and in different columns. Let the vector $c\in\mathbb{R}^d$ be the vector whose $n^{th}$ component, denoted $c_n$, corresponds to the row of $A$ such that $a_{c_n,n}=1$. Since $A$ is not a Permutation Matrix, we know that not every row sums to 1. In particular, there exists $1\leq m\neq n\leq d$ such that $c_m=c_n$. However, there are $d$ rows and $d$ columns, which implies there exists some $1\leq n_0\leq d$ such that $n_0$ does not appear as a component in the vector $c$. Thus, in the matrix $A$, the row $n_0$ consists of all zeros and hence the determinant is zero. Therefore, $0$ is an eigenvalue of $A$.
\end{proof}
This immediately yields the following corollary. Note that the identity matrix is considered to be a permutation matrix.
\begin{corollary}\label{InvertCor}
    The only invertible elements of $PL_d$ for $d>1$ are the Permutation Matrices.
\end{corollary}
\begin{remark}
    Corollary \ref{InvertCor} implies that the only elements of the Stochastic Group that are in $PL_d$ are the Permutation Matrices.
\end{remark}

\subsection{Left Stochastic Matrices}

Lastly, we would like to visit the idea of Stochastic Matrices and how it relates to PLM. Recall that a matrix $A$ with non-negative entries $a_{i,j}$ is Left (Column) Stochastic if for each $1\leq j\leq d$, it follows that $\displaystyle\sum_{i=1}^d a_{i,j}=1$. Similarly, $A$ is Right (Row) Stochastic if for each $1\leq i\leq d$, it follows that $\displaystyle\sum_{j=1}^d a_{i,j}=1$. Lastly, we say that a matrix is Doubly Stochastic if it is both Left and Right Stochastic (\cite{Linear}).

There is a classical result, called Birkhoff-Von Neumann's Theorem (\cite{birk}), which says that a matrix $A$ is Doubly Stochastic if and only if there exists $m\geq1$ and $\lambda_i\in[0,1]$ for $1\leq i\leq m$ such that $\displaystyle\sum_{i=1}^m\lambda_i=1$ and
$A=\displaystyle\sum_{i=1}^m\lambda_iP_i$, where $P_i$ are Permutation Matrices. Recall that the condition that there exists $m\geq1$ and $\lambda_i\in[0,1]$ for $1\leq i\leq m$ such that $\displaystyle\sum_{i=1}^m\lambda_i=1$ and
$A=\displaystyle\sum_{i=1}^m\lambda_iP_i$, where $P_i$ are Permutation Matrices means that every Doubly Stochastic Matrix is a convex combination of Permutation Matrices. We would like to state a similar result for PLM regarding Left Stochastic Matrices. First, notice the following comes immediately from the definition of left stochastic with the definition of PLM.
\begin{lemma}
    Every Permutation-Like Matrix is a Left Stochastic Matrix.
\end{lemma}
Next, we would like to prove a similar result as Birkhoff-Von Neumann's Theorem, but for Left Stochastic Matrices.
\begin{lemma}\label{ConvexStoch}
    Let $\lambda_i\in [0,1]$ for $1\leq i\leq k$ such that $\displaystyle\sum_{i=1}^k\lambda_i=1$ and let $A_i$ for $1\leq i\leq k$ be Left Stochastic Matrices. Then, $\lambda_1A_1+\lambda_2A_2+\ldots+\lambda_kA_k$ is a Left Stochastic Matrix.
\end{lemma}
\begin{proof}
    For every $1\leq i\leq k$, notice that the sum of every column of $A_i$ is 1 by definition of left stochastic, so the sum of every column of $\lambda_iA_i$ is $\lambda_i$. Further, since matrix addition is pointwise, we have that the sum of every column of $\lambda_1A_1+\lambda_2A_2+\ldots+\lambda_kA_k$ is $\lambda_1+\lambda_2+\ldots+\lambda_k=1$. Therefore, $\lambda_1A_1+\lambda_2A_2+\ldots+\lambda_kA_k$ is a Left Stochastic Matrix.
\end{proof}
Notice that Lemma \ref{ConvexStoch} implies that the convex combination of PLM is Left Stochastic. We will utilize this in the following theorem.
\begin{theorem}\label{LeftBvN}
    Let $A$ be a $d\times d$ matrix. Then, $A$ is left stochastic if and only if $A$ is a convex combination of Permutation-Like Matrices.
\end{theorem}
\begin{proof}
    The reverse implication follows immediately from Lemma \ref{ConvexStoch}. Thus, we will focus on the forward implication. 

    Given a $d\times d$ matrix $B$ and $1\leq j\leq d$, let $r(B)\in k^d$ be the vector that gives the row of the first strictly positive entry in each column of $B$. Let $r(B)_j$ be the $j^{th}$ component of $r(B)$, i.e. the row of the first strictly positive entry in column $j$. Lastly, let $P(B)\in PL_d$ be the PLM such that for each $1\leq j\leq d$, $P(B)$ has a 1 in $r(B)_j$ row and $j^{th}$ column, with zeros elsewhere. For example, if we take $B$ to be the Left Stochastic Matrix
    \[B=\begin{pmatrix}
        0.1&0&0.2\\
        0.9&0.5&0.8\\
        0&0.5&0
    \end{pmatrix}\]
    then, $P(B)$ is given by
    \[P(B)=\begin{pmatrix}
        1&0&1\\
        0&1&0\\
        0&0&0
    \end{pmatrix}.\]
    
    Let $A$ be a $d\times d$ Left Stochastic Matrix with entries $a_{i,j}$. In particular, we know that for all $1\leq j\leq d$ that $\displaystyle\sum_{i=1}^da_{i,j}=1$. If $A\in PL_d$, we are done. Suppose that $A$ is not a PLM.

    First, let $\lambda_{1}=min_{1\leq j\leq d}\{a_{r(A)_j,j}|1\leq j\leq d\}$. Notice that $\lambda_1<1$ since the columns are non-negative, sum to 1, and $A$ is not a PLM so there exists at least two strictly positive terms in some column. Then, $A_1=A-\lambda_1P(A)$ is a $d\times d$ matrix with at least 1 more zero than $A$. Notice that the sum of the columns of $A_1$ are all $1-\lambda_1$. 

    Next, suppose that for some $n\geq1$, we have $A_n=A-\sum_{i=1}^n\lambda_iP_i$ for PLM $P_i$. Take $\lambda_{n+1}=min_{1\leq j\leq d}\{a_{r(A_{n})_j,j}|1\leq j\leq d\}$. Notice that the columns of $A_n$ sum to $1-\sum_{i=1}^n\lambda_i$, so $\lambda_{n+1}\leq 1-\sum_{i=1}^n\lambda_i$. Now, if we take $A_{n+1}=A_n-\lambda_{n+1}P(A_n)$, we have at least 1 more zero in $A_{n+1}$ than $A_n$. Further, the columns of $A_{n+1}$ sum to $1-\sum_{i=1}^{n+1}\lambda_i,$ strictly less than the sum of each column of $A_n$. 
    
    Either $A_{n+1}$ is the zero matrix or we can iterate this process. Since there are finitely many positions, this process terminates. Further, we are reducing each column by $0<\lambda_{n+1}<1$ each time, so this process will terminate for all columns at the same step. Lastly, because we are taking $\lambda_{n+1}$ to be a minimum across the first strictly positive entries in each column (ordering the rows using the usual ordering on the integers), we never end up with any negative entries. Therefore, there exists some $N\geq1$ such that
    \[A-\sum_{i=1}^N\lambda_iP_i=0\]
    where $P_i\in PL_d$ and $0<\lambda_i<1$ for all $1\leq i\leq N$. Notice that $\sum_{i=1}^N\lambda_i=1$, since following the last step the columns sum to $0=1-\sum_{i=1}^N\lambda_i$. Thus, by Lemma \ref{ConvexStoch}, we know that $\sum_{i=1}^N\lambda_iP_i$ is a Left Stochastic Matrix and
    \[A=\sum_{i=1}^N\lambda_iP_i.\]
\end{proof}
\begin{remark}
    Stated another way, this theorem says that we know that the polytope in $k^{d^2}$ formed by Left Stochastic Matrices has $PL_d$ as its vertices. A similar proof could be used to show that the polytope in $k^{d^2}$ by Right Stochastic Matrices has its vertex set as the set of matrices whose transpose is a PLM. We will not go in this direction in this paper because we are focused explicitly on PLM. 
\end{remark}
\section{Remarks on Permutation-Like Matrices}
We would like to end on a handful of remarks and open questions concerning PLM.

First, notice that the motivation for studying PLM was to determine a general pattern for leg identifying multiplication on edge $d$-partitions of $K_{2d}$. However, as seen in the case $d>3$, the descriptions of the multiplication of PLM although explicit are not as straightforward to use in a general rule. For that reason, we have the following open problem. 
\begin{problem}
    Obtain an explicit description of the multiplication of edge $d$-partitions of $K_{2d}$ using the results of this paper.
\end{problem}

Also, recall that the eigenvalues of the Permutation Matrices are all roots of unity. It is interesting that the eigenvalues of PLM for $d=2$ and $d=3$ are roots of unity and 0 (for PLM that are not Permutation Matrices or the identity). This fits into the notion that PLM generalize the Permutation Matrices, as the eigenvalues of the Permutation Matrices are roots of unity. Therefore, we have the following open conjecture.
\begin{conjecture}
    Let $A\in PL_d$ be a PLM for $d>1$. Then, the eigenvalues of $A$ are zeros or roots of unity. Further, zero is an eigenvalue if and only if $A$ is not a Permutation Matrix or the identity matrix.
\end{conjecture}
Notice that the second half of this conjecture has been answered, whereas the former part has only been answered for $d\leq3$.

Third, a constructive formulation was obtained for Theorem \ref{LeftBvN} and a similar formulation can be obtained for the Right Stochastic version (as previously remarked). One open question follows since the polytope (Birkoff Polytope) formed by Doubly Stochastic Matrices is the intersection of the polytopes formed by Left and Right Stochastic Matrices. 
\begin{problem}
    Given a doubly stochastic matrix $A$, determine if its convex combination of PLM constructed in Theorem \ref{LeftBvN} and convex combination of transposes of PLM constructed in a similar way can be used to write $A$ as a convex combination of Permutation Matrices (as stipulated in Birkhoff-Von Neumann's Theorem).
\end{problem}

Fourth, recall the Stochastic Group (introduced in \cite{StochGrp}) is defined as the group of $d\times d$ invertible Left Stochastic Matrices. We know from Theorem \ref{LeftBvN} that every Left Stochastic Matrix can be written as a convex combination of PLM. This leads to the following problem.
\begin{problem}
    Obtain a computational formula for writing elements of the Stochastic Group as convex combinations of PLM.
\end{problem}

Fifth, $A$ is a $d\times d$ pre-row PLM for $d=2$ or $d=3$ if and only if $A$ is a CPLM with leading element zero and PLC $R_m^{(d-1)}$ for some $1\leq m\leq d-1$. The reverse implication of this statement holds for $d>3$, but the forward implication is still an open problem.
\begin{problem}
    For all $d>1$, show that every $d\times d$ pre-row PLM is a CPLM with leading element zero and PLC $R_m^{(d-1)}$ for some $1\leq m\leq d-1$.
\end{problem}

Lastly, one nice property of Permutation Matrices is the identification with the symmetric group. Further, we know PLM are closed under multiplication, have an identity, but are not necessarily invertible (thus has a monoid structure). This leads to the following problem.
\begin{problem}
    Determine if there exists a monoid which generalizes the symmetric group and is isomorphic to $PL_d$.
\end{problem}

\section*{Acknowledgment}
We thank Jeremy Case for discussion regarding this paper as well as his comments towards the presentation of the paper.

\bibliographystyle{amsalpha}

\end{document}